\documentclass[11p,leqno]{amsart}
\usepackage{amsmath}
\usepackage{amsfonts}
\usepackage{amssymb}
\usepackage{graphicx}
\usepackage{color}
\newtheorem{theorem}{Theorem}[section]
\newtheorem*{theorem*}{Theorem}
\newtheorem{lemma}{Lemma}[section]
\newtheorem{corollary}[theorem]{Corollary}
\newtheorem{proposition}{Proposition}[section]

\newtheorem{definition}[theorem]{Definition}

\numberwithin{equation}{section}

\begin{document}
\title[Heat equation along an extended Ricci flow]{\bf A Gaussian upper bound of the conjugate heat equation along an extended Ricci flow}

\author{Xian-Gao Liu}

\address{School of Mathematic Sciences, Fudan University, Shanghai, 200433, China}
\email{xgliu@fudan.edu.cn}

\author{ Kui Wang}


\address{School of Mathematic Sciences, Fudan University, Shanghai, 200433, China}
\email{kuiwang09@fudan.edu.cn}

\date{}

\maketitle

\begin{abstract} \quad In this paper, we derive a Sobolev inequality along an extended Ricci flow
and prove a point-wise Guassian type bound for the fundamental solutions of the conjugate heat equation
under the flow.
\end{abstract}

\section{Introduction}
\setcounter{equation}{0}
Let $M^n$ be a $n$ dimensional closed smooth manifold and assume $n\ge 3$.
In \cite{BL}, B. List studied a Ricci flow, coupled with a harmonic flow
\begin{equation}\left
\{\begin{array}{l}
\large{\partial_t g(x,t)=-2\text{Ric}(g(x,t))+4d\phi(x,t)\otimes d\phi(x,t)},\\
\large{\partial_t \phi(x,t)=\triangle_{g(x,t)}\phi,}
\end{array}\right.\label{O1}
\end{equation}
where $g(x,t)$ is a family of Riemannian metrics, and $\phi(x,t)$ is a function on $M$ for any fixed $t$. This flow is also called
Ricci-Harmonic (RH) flow (c.f.\cite{BL,MH,MR,LK}).
If $\phi$ is a constant function, the system (\ref{O1}) degenerates to Hamilton's Ricci flow discussed widely recently,
see for example the book \cite{CLN} and celebrated papers \cite{Ham1, Ham2, Ham3, N2,P1, LJ}. The Ricci-Harmonic flow 
 is very useful in general relativity. The stationary solutions (only depending on spatial variable $x$)
of (\ref{O1}) are solutions to the following
static Einstein vacuum system
\begin{equation*}\left
\{\begin{array}{l}
\large{\text{Ric}=2d\phi\otimes d\phi,}\\
\large{\Delta \phi=0.}
\end{array}\right.
\end{equation*}
Similarly as Ricci flow, corresponding theories for the extended Ricci flow system was established in \cite{BL}, such as the
short time existence and the bounds of curvature tensor $Rm$ and $\phi$.

Let $g$ be a Riemannian metric and $\phi$ be a smooth function on $M$.
For the sake of conveniences, we denote as in \cite{BL} the symmetric tensor field $Sy\in Sym_2(M)$ and its trace $S:=g^{ij}S_{ij}$  by
\begin{equation*}
S_{ij}:=R_{ij}-2\partial_i \phi\partial_j \phi\text{\quad and\quad}S:=R-2|d\phi|^2.
\end{equation*}
Where $R$ denotes the scale curvature of $(M, g)$.
The summation convention of summing over repeated induces is used here and throughout the paper. Then the RH flow can be written as
 \begin{equation*}\left
           \{\begin{array}{l}
           \large{\partial_t g=-2Sy,}\\
           \large{\partial_t \phi=\triangle \phi.}\\
           \end{array}\right.
 \end{equation*}

It is well-known that the Sobolev inequality contains a host of analytical and geometric information (c.f.\cite{Heb, Sal, LS}) including  non-collapsing,  isoperimetric inequalities and so on. It is also an important tool in studying elliptic and parabolic differential equations on manifolds (c.f.\cite{Sal}). In \cite{KZ, Z1, Z2, Z3}, via the monotonicity of Perelman's $W$ entropy, some uniform Sobolev inequalities were proved along Ricci flow. As a consequence,  a long time non-collapsing result was established, which generalizes Perelman's short time result.
In \cite{Z2}, under the assumption that Ricci curvature is non-negative and the injectivity is bounded from below, Zhang proved a global upper bound for the fundamental solution of a heat equation introduced by Perelman under backward Ricci flow, i.e.
\begin{equation*}\left
\{\begin{array}{l}
\large{\partial_t g=-2\text{Ric},}\\
\large{ \Delta u- \partial_t u-Ru=0.}
\end{array}\right.\
\end{equation*}
Along flow (\ref{O1}), we consider the following conjugate heat equation
\begin{equation}
\partial_t u(x,t)+\Delta u(x,t)-S(x,t)u(x,t)=0.\label{O2}
\end{equation}
In \cite{MH, LK}, some point-wise gradient estimates for the positive solutions of (\ref{O1}) was obtained,
 which can be viewed as the Li-Yau estimate for the parabolic kernel of the Schr\"{o}dinger operator in \cite{LY, N1, N2}.

The main goal of this paper is to establish some certain Sobolev inequalities under system (\ref{O1}) and a global upper bound for the fundamental solution of heat equation (\ref{O2}) under the extended Ricci flow.
Via the monotonicity of the entropies, we obtain the following Sobolev inequality.
\begin{theorem}\label{th1}
Let ($g(x,t), \phi(x,t)$)
 be a solution of the system (\ref{O1}) for $t\in[0,T_0)$ with initial metric $g_0$,
 where $T_0\leq\infty$ is the life span of (\ref{O1}).
Let  $A_0$ and $B_0$ be positive numbers such that the following $L^2$ Sobolev inequality holds initially,
i.e. for any $v\in W^{1,2}(M, g_0)$,
$$
\Big(\int_M v^{\frac{2n}{n-2}}d\mu\big(g_0\big)\Big)^{\frac{n-2}{n}}\leq A_0 \int_M\big(|\nabla v|^2+\frac{1}{4}Sv^2\big)d\mu\big(g_0\big)+B_0 \int_M v^2d\mu\big(g_0\big).
$$
Then for all $v\in W^{1,2}(M,g(t))$,we have
\begin{eqnarray}
\Big(\int_M v^{\frac{2n}{n-2}}d\mu\big(g(t)\big)\Big)^{\frac{n-2}{n}}\leq A(t) \int_M \big(|\nabla v|^2+\frac{1}{4}Sv^2\big)d\mu\big(g(t)\big)
+B(t)\int_M v^2d\mu\big(g(t)\big).\nonumber\\
\label{S2}
\end{eqnarray}
Where $A(t)$ and $B(t)$ are positive constants depending only on $A_0$, $(1+t)B_0$ and $n$.
\end{theorem}
By the above Sobolev inequality , combining with Morse's iteration and Davies heat kernel estimate,
we prove following Gussian type upper bound for the fundamental solutions of (\ref{O2})
and the bound does not depend on the lower bound of injective radius, but depends on 
the the first eigenvalue of the entropy, which is different from Zhang's result 
in \cite {Z2}. More precisely,  we prove
\begin{theorem}\label{th2}
Let ($g(x,t), \phi(x,t)$) be a smooth solution of system (\ref{O1}) in $M\times[0,T]$ and
 $G(x,t;y,T)$ be a fundamental solution of the following backward conjugate heat equation (\ref{O2}), that is
\begin{equation*}\left
\{\begin{array}{l}
\large{\triangle_x G(x,t;y,T)+\partial_t G(x,t;y,T)-S(x,t)G(x,t;y,T)=0},\text{\quad\quad}0\leq t<T;\\
\large{G(x,t;y,T)=\delta(x,y)},\text{\quad\quad}t=T.\\
\end{array}\right.\end{equation*}
Assume further that
$Sy\geq0$ and the first eigenvalue $\lambda_0$ of entropy $\inf_{\|v\|_2=1}\int_M(4|\nabla v|^2+Sv^2)d\mu\big(g_0\big)$
is positive.
Then for any $t\in(0,T)$, and $x,y\in M$, we have the following  estimates
\begin{equation}
G(x,t;y,T)\leq \frac{c}{|B(y,\sqrt{T-t},T)|_T}\exp{\frac{-c_1d^2(x,y,T)}{T-t}},\label{BRF1}
\end{equation}
where $c_1$ is a constant depending only on dimension $n$) and  $c$ is a constant
depending on dimension $n$, $\lambda_0$  and initial metric $g_0$.
Here $d(x,y,T)$ denotes the distance between $x$ and $y$ with respect to metric $g(T)$,
$B(y,\sqrt{T-t},T)$ denotes the geodesic ball centered at $y$ with radius $\sqrt{T-t}$,
and  $|B(y,\sqrt{T-t},T)|_T$ denotes the volume of the ball $B(y,\sqrt{T-t},T)$ with
respect to metric $g(T)$.
\end{theorem}
The rest of the paper is organized as follows. The evolution formulas of entropies under system (\ref{O1}) are given in section 2. Corresponding Sobolev inequalities along the extended Ricci flow system are derived in section 3. In section 4, we prove Theorem \ref{th2}.

\section{Entropies of the extended Ricci flow}
\setcounter{equation}{0}
In this section, we recall the definitions of entropies via corresponding conjugate heat equation just as Perelman  has done in Ricci flow in \cite{P1}. Through direct computations, we obtain the monotone quantities of the entropies. Although the monotonicity of the entropies has been obtained  in \cite{BL},  for the completeness we will give  a different computation to obtain the evolution equations of entropies without using the entropies' invariance under diffeomorphisms.

Let $u(x,t)$ be a positive solution to the conjugate heat equation (\ref{O2})
\begin{equation*}
H^*u:=\triangle u-Su+\partial_t u=0.\end{equation*}
From the equation (\ref{O2}) and the evolution equations (\ref{O1}), it follows easily that
$$
\frac{d}{dt}\int_Mu(x,t)d\mu(g(t))=\int_M(\partial_t-S)ud\mu(g(t))=\int_MH^*ud\mu(g(t))=0.
$$
Here we used the fact that $M$ is  closed. Therefore, we assume that $u(x,t)$ satisfies
\begin{equation}
\int_Mu(x,t)d\mu(g(t))=1
\label{E1}\end{equation}
 for any $t\in[0,T]$.
 
Via the positive solution $u$ of (\ref{O2}), the entropies are defined (see for example \cite{BL}) as follows
\begin{definition}
$F$ entropy is defined as the following integration
\begin{equation}
F(t):=\int_M\left(Su+\frac{|\nabla u|^2}{u}\right)d\mu(g(t)),\label{Fentropy}
\end{equation}
and $W$ entropy is defined by
\begin{equation}
W(t):=\int_M\left[\tau(Su+\frac{|\nabla u|^2}{u})-u\ln u-\frac{n}{2}\ln( 4\pi\tau)u -nu \right]d\mu(g(t)),\label{Wentropy}
\end{equation}
where $\tau$ is a scaling factor satisfied $\frac{d\tau}{dt}=-1$.
\end{definition}
In order to simplify the computations in this paper, we introduce a function $f(x,t)$ given by:
\[u(x,t)=\frac{e^{-f}}{(4\pi\tau)^{\frac{n}{2}}},\]
where $\frac{d\tau}{dt}=-1$. Then it follows that
\begin{equation}
f=-\ln u-\frac{n}{2}(\ln 4\pi\tau).\label{Euf}
\end{equation}

With the above preparations, we  now give a direct calculation of the following monotonicity formula.
\begin{proposition}\label{p1}
(see also Lemma 3.4 and Theorem 6.1 in \cite{BL}) Let $(g, \phi)(t)$ be a solution of (\ref{O1}) and $u(x,t)$
be a positive solution of (\ref{O2}). Then \\
(I) $F$ entropy is nondecreasing in $t$. More precisely,
\begin{equation}
\frac{d}{dt}F(t)=2\int_M\left(|Sy+\nabla^2f|^2+2|\triangle \phi-d\phi(\nabla f)|^2\right)ud\mu(g(t))\ge 0.\label{Fem}
\end{equation}
(II) $W$ entropy is nondecreasing in $t$. More precisely,
\begin{equation}
\frac{d}{dt}W(t)=\int_M\left(2\tau|Sy+\nabla^2f-\frac{g}{2\tau}|^2+4\tau|\triangle \phi-d\phi(\nabla f)|^2\right)ud\mu(g(t))
\ge0.\label{Wem}
\end{equation}
\end{proposition}
\begin{proof}
We start by direct calculations that
\begin{equation}
H^*(u\ln u)=\frac{|\nabla u|^2}{u}+Su.\label{The1}
\end{equation}
and
\begin{eqnarray}
H^*(\frac{|\nabla u|^2}{u}+S u)&=&\frac{2}{u}(\frac{u_i u_j}{u}-u_{ij})^2+\frac{2S_{ij}u_i u_j}{u}+\frac{2R_{ij}u_i u_j}{u}\nonumber\\
&&+4\langle \nabla u, \nabla S\rangle+2u\triangle S+2(|Sy|^2+2|\triangle \phi|^2)u,\label{ThE1}
\end{eqnarray}
here we used the evolution equation $\partial_t S=\triangle S+2|S_{ij}|^2+4|\triangle \phi|^2$
(see for example \cite{BL, LK}).

Taking time derivative yields
\begin{eqnarray*}
\frac{d}{dt}F&=&\frac{d}{dt}\int_M(Su+\frac{|\nabla u|^2}{u})d\mu
=\int_M(\partial_t-S)(Su+\frac{|\nabla u|^2}{u})d\mu\nonumber\\
&=&\int_M H^*(Su+\frac{|\nabla u|^2}{u})d\mu.
\end{eqnarray*}
Substituting (\ref{ThE1}) to the above equality, we get
\begin{eqnarray}
\frac{d}{dt}F&=&\int_M\big[\frac{2}{u}(\frac{u_i u_j}{u}-u_{ij})^2+\frac{2S_{ij}u_i u_j}{u}+\frac{2R_{ij}u_i u_j}{u}\nonumber\\
&&+4\langle \nabla u,\nabla S\rangle+2u\triangle S+2(|Sy|^2+2|\triangle \phi|^2)u\big]d\mu.\label{ThE5}
\end{eqnarray}
By integration by parts and the contracted second Bianchi identity, we have
\begin{eqnarray}
\int_M\langle \nabla u,\nabla S\rangle d\mu&=&\int_M\langle \nabla u, \nabla (R-2|d\phi|^2)\rangle d\mu
=\int_M\big(2u_i\nabla_j R_{ij}-4u_i\phi_j\phi_{ij}\big)d\mu\nonumber\\
&=&\int_M\big(-2u_{ij}R_{ij}+4u_{ij}\phi_j\phi_i+4u_i\phi_i\triangle \phi \big)d\mu\nonumber\\
&=&\int_M\big(-2u_{ij}S_{ij}+4u_i\phi_i\triangle \phi \big) d\mu\label{ThE6}.
\end{eqnarray}
Substituting (\ref{ThE6}) into (\ref{ThE5}), we obtain
\begin{eqnarray*}
\frac{d}{dt}F&=&\int_M\big[\frac{2}{u}(\frac{u_i u_j}{u}-u_{ij})^2+\frac{2S_{ij}u_i u_j}{u}+\frac{2R_{ij}u_i u_j}{u}\\
&&+2\langle \nabla u,\nabla S\rangle+2(|Sy|^2+2|\triangle \phi|^2)u\big]d\mu\\
&=&\int_M\big[\frac{2}{u}(\frac{u_i u_j}{u}-u_{ij})^2+\frac{2S_{ij}u_i u_j}{u}+\frac{2R_{ij}u_i u_j}{u}\\
&&-4u_{ij}S_{ij}+8u_i\phi_i\triangle \phi+2(|Sy|^2+2|\triangle \phi|^2)u\big]d\mu.
\end{eqnarray*}
Replacing $u$ with $f$ in the above equality, we deduce
\begin{eqnarray*}
\frac{d}{dt}F&=&\int_M\big[|f_{ij}|^2+2S_{ij}f_{ij}+|S_{ij}|^2+2|\triangle \phi|^2+2|2d\phi(\nabla f)|^2\\
&&-4\triangle \phi (d\phi (\nabla f))\big]d\mu\\
&=&2\int_M\big[|Sy+\nabla^2f|^2+2|\triangle \phi-d\phi(\nabla f)|^2\big]ud\mu,
\end{eqnarray*}
which gives formula (\ref{Fem}).

From the definition of $W$ entropy, it follows that
\begin{equation*}
\frac{d}{dt}W(t)=\int_MH^*(\tau(\frac{|\nabla u|^2}{u}+S u))-H^*(u\ln u)-\frac{n}{2}H^*(u\ln\tau)d\mu.
\end{equation*}
Substituting (\ref{The1}) and (\ref{ThE1}) to the above equality, we get
\begin{eqnarray}
\frac{d}{dt}W&=&\tau\frac{d}{dt}F-2F+\frac{n}{2\tau}\nonumber\\
&=&\tau\frac{d}{dt}F-2\int_M(Su+|\nabla f|^2u)d\mu(g(t))+\frac{n}{2\tau}.\label{The2}
\end{eqnarray}
By the definition of $f$ and integrations by parts, we deduce
\begin{eqnarray*}
\int_M(Su+\frac{|\nabla u|^2}{u})d\mu(g(t))=\int_M(Su-\nabla u\nabla f)d\mu
=\int_M(Su+ \Delta f u)d\mu.
\end{eqnarray*}
Substituting the above equality and equality (\ref{Fem}) into (\ref{The2}), we have
\begin{eqnarray*}
\frac{d}{dt}W&=&2\tau\int_M\left(|Sy+\nabla^2f|^2+2|\triangle \phi-d\phi(\nabla f)|^2\right)ud\mu
-2\int_M(S+ \Delta f)ud\mu+\frac{n}{2\tau}\\
&=&\int_M2\tau\big[|Sy+Hess(f)-\frac{g}{2\tau}|^2+2|\triangle \phi-d\phi(\nabla f)|^2)\big]ud\mu.
\end{eqnarray*}
Thus we complete the proof.
\end{proof}

Similarly as the Ricci flow, one can define a family of generalized $W$ entropy along the extended Ricci flow by
\begin{eqnarray}
W(a,t)&:=&\int_M\left(\frac{a^2\tau}{2\pi}(Su+\frac{|\nabla u|^2}{u})-u\ln u-\frac{n}{2}\ln (4\pi\tau)u -nu\right)d\mu(g(t))\nonumber\\
&=&\int_M\left(\frac{a^2\tau}{2\pi}(S+|\nabla f|^2)+f-n\right)ud\mu(g(t)).\label{GWentropy}
\end{eqnarray}
Here the second equality is due to the relations between $u$ and $f$ given by (\ref{Euf}). The more applications of generalized entropy can be found in \cite{LJ}. Using the calculations in \cite{KZ}, we get the following monotonicity formula of generalized $W$ entropy.
\begin{proposition}
Let $M$ be a closed Riemannian manifold, $(g, \phi)(t)$ be a solution of (\ref{O1}) and $u(x,t)$ be a positive solution of (\ref{O2}). Then the generalized $W(a,t)$ entropy is nondecreasing in $t$ and the following inequality holds:
$$
\frac{d}{dt}W(a,t)\geq\frac{a^2\tau}{\pi}\int_M\left(|Sy+Hess(f)-\frac{g}{2\tau}|^2+2|\triangle \phi-d\phi(\nabla f)|^2\right)ud\mu.
$$
\end{proposition}
Since the proof of this proposition is the same as that in Ricci flow case, we omit the details here.
One can find details in  Theorem 4.1 in \cite{KZ}.

\section{ Sobolev inequalities under the extended Ricci flow}
\setcounter{equation}{0}
In this section, we mainly use the monotonicity of $W$ entropy to derive a uniform Sobolev inequality along the system (\ref{O1}) which will be used in the next section.

To prove Theorem \ref{th1}, we need the following lemma first, which shows the relations between the logarithmic Sobolev inequality, the $W^{1,2}$ Sobolev inequality and the so-called ultracontractivity of the heat semigroup of the associated Schr\"{o}dinger operator. The proof of this lemma is more or less standard. One can consult more details for example in Theorem 4.2.1 in \cite{Z3}.
\begin{lemma}[See for example in \cite{Z3}]\label{lm1}
Let $(M^n, g)$ be a  closed Riemannian manifold ($n\geq3$). Then the following inequalities are equivalent (up to constants).\\
(I) Sobolev inequality: there exists positive constants $A$ and $B$ such that, for all $v\in W^{1,2}(M)$,
\begin{equation*}
(\int_M v^{\frac{2n}{n-2}}d\mu)^{\frac{n-2}{n}}\leq A\int_M|\nabla v|^2d\mu+B\int_Mv^2d\mu;
\end{equation*}
(II) Log-Sobolev inequality: for all $v\in W^{1,2}(M)$ such that $\|v\|_2=1$ and all $\epsilon >0$,
\begin{equation*}
\int_Mv^2\ln v^2 d\mu\leq \epsilon^2\int_M|\nabla v|^2d\mu-\frac{n}{2}\ln\epsilon^2+BA^{-1}\epsilon^2+\frac{n}{2}\ln\frac{nA}{2e};
\end{equation*}
(III) Heat kernel upper bound: for all $t>0$,
\begin{equation*}
G(x,t;y)\leq\frac{(nA)^{\frac{n}{2}}}{t^{\frac{n}{2}}}e^{A^{-1}Bt};
\end{equation*}
(IV) Nash inequality: for all  $v\in W^{1,2}(M)$,
\begin{equation*}
\|v\|^{2+\frac{4}{n}}_2\leq (A\|\nabla v\|^2_2+B\|v\|^2_2)\|v\|^{\frac{4}{n}}_1.
\end{equation*}
\end{lemma}
Due to Lemma \ref{lm1}, to prove Theorem \ref{th1}, we only need to prove some log-Sobolev inequalities for any $t\in[0,T_0)$. By the monotonicity of $W$ entropy, we obtain the following log-Sobolev inequality.
\begin{lemma}[Log-Sobolev Inequality]\label{lm2}
Under the assumptions of Theorem \ref{th1}. Then for any $t\in[0,T_0)$,  $v\in W^{1,2}(M,g(t))$ with $\int_M v^2d\mu(g(t))=1$ and any $\epsilon>0$, we have
\begin{eqnarray}
\int_M v^2\ln v^2d\mu(g(t))&\leq&\epsilon^2\int_M\big(4|\nabla v|^2+Sv^2\big)d\mu(g(t))-n\ln\epsilon\nonumber\\
&&+(t+\epsilon^2)B_0A_0^{-1}+\frac{n}{2}\ln(\frac{nA_0}{2e}).\label{S3}
\end{eqnarray}
\end{lemma}
\begin{proof}
For a fixed $t_0\in[0,T_0)$ and any $\epsilon>0$, we set
$$\tau(t)=\epsilon^2+t_0-t.$$
Recall that $W$ entropy is defined by
\[W(g,f,t)=\int_M\big(\tau(S+|\nabla f|^2)+f-n\big)ud\mu(g(t)).\]
From the monotonicity of the $W$ entropy in Proposition \ref{p1}, we deduce
\begin{equation}
\inf_{\int_M u d\mu(g(t_0))=1}W(g(t_0),f,\epsilon^2)\geq \inf_{\int_M u_0d\mu(g(0))=1}W(g(0),f_0,t_0+\epsilon^2),
\label{21}\end{equation}
one can find a more detailed proof of this property in section 3 of \cite{P1}. Here $f_0$ and $f$ are given by the formulas
\[
u_0=\frac{e^{-f_0}}{\big(4\pi(t_0+\epsilon^2)\big)^{\frac{n}{2}}},\quad u=\frac{e^{-f}}{(4\pi\epsilon^2)^{\frac{n}{2}}}.
\]
Using these notations we  rewrite (\ref{21}) as
\begin{eqnarray*}
&&\inf_{\int ud\mu(g(t_0))=1}\int_M\left(\epsilon^2(S+|\nabla \ln u|^2)-\ln u-\frac{n}{2}\ln4\pi\epsilon^2\right)ud\mu(g(t_0))\\
&\geq&\inf_{\int u_0d\mu(g(0))=1}\left(\int_M\big((\epsilon^2+t_0)(S+|\nabla
\ln u_0|^2)-\ln u_0
-\frac{n}{2}\ln4\pi(t_0+\epsilon^2)\big)u_0d\mu(g(0))\right).
\end{eqnarray*}
Set $v=\sqrt{u}$ and $v_0=\sqrt{u_0}$, the above inequality yields
\begin{eqnarray}
&&\inf_{\int v^2d\mu(g(t_0))=1}\int_M\big[\epsilon^2(Sv^2+4|\nabla v|^2)-v^2\ln v^2\big]d\mu(g(t_0))-\frac{n}{2}\ln\epsilon^2\nonumber\\
&\geq&\inf_{\int v_0^2d\mu(g(0))=1}\int_M\left((\epsilon^2+t_0)(Sv_0^2+4|\nabla
v_0|^2)-v_0^2\ln v_0^2\right)d\mu(g(0))
-\frac{n}{2}\ln(t_0+\epsilon^2).\nonumber\\\label{22}
\end{eqnarray}
Notice that $\ln x$ is a concave function and $\int v_0^2d\mu(g(0))=1$, thus applying Jensen's inequality we deduce
\[\int_Mv_0^2\ln v_0^{q-2}d\mu(g(0))\leq\ln \int v_0^{q-2}v_0^2d\mu(g(0)),\]
i.e.
$$
\int_Mv_0^2\ln v_0^2d\mu(g(0))\leq\frac{n}{2}\ln\|v_0\|_q^2,
$$
where $q=\frac{2n}{n-2}$. By the assumption that the Sobolev inequality holds for the initial time $t=0$, we have
\[\int_Mv_0^2\ln v_0^2d\mu(g(0))\leq\frac{n}{2}\ln\left(A_0\int_M(4|\nabla v_0|^2+Sv_0^2)d\mu(g(0))+B_0\right).\]
By the inequality \[\ln z\leq yz-\ln y-1,\]
for any $y,z>0$, we have
\begin{equation*}
\int_Mv_0^2\ln v_0^2d\mu(g(0))\leq\frac{n}{2}y\Big(A_0\int_M(4|\nabla v_0|^2+Sv_0^2)d\mu(g(0))+B_0\Big)-\frac{n}{2}\ln y-\frac{n}{2}.
\end{equation*}
In the above inequality we choose $y=2\frac{t_0+\epsilon^2}{n A_0}$,
we then deduce
\begin{eqnarray*}
\int_Mv_0^2\ln v_0^2d\mu(g(0))&\leq&(t_0+\epsilon^2)\int_M(4|\nabla v_0|^2+Sv_0^2)d\mu(g(0))\\
&&+\frac{(t_0+\epsilon^2)B_0}{A_0}-\frac{n}{2}\ln\frac{2(t_0+\epsilon^2)}{nA_0}-\frac{n}{2}.
\label{S5}\end{eqnarray*}
Substituting the above inequality to the right-hand side of (\ref{22}), we reach the log-Sobolev inequality:
\begin{eqnarray*}
\int_M v^2\ln v^2d\mu(g(t_0))&\leq&\epsilon^2\int_M\big(4|\nabla v|^2+Sv^2\big)d\mu(g(t_0))-n\ln\epsilon\\
&&+(t_0+\epsilon^2)B_0A_0^{-1}+\frac{n}{2}\ln(\frac{nA_0}{2e}).
\end{eqnarray*}
Thus we have completed the proof of the log-Sobolev inequality (\ref{S3}), hence Theorem \ref{th1}.
\end{proof}

Since $(M, g_0)$ is a closed Riemannian manifold, the Sobolev inequality holds as described in section 4.1 in \cite{Z3}, i.e. for any $v\in W^{1,2}(M)$, there exist positive constants $A$ and $B$  depending only on the initial metric $g_0$ such that
\begin{eqnarray}
\Big(\int_M v^{\frac{2n}{n-2}}d\mu\big(g(0)\big)\Big)^{\frac{n-2}{n}}&\leq&A\int_M|\nabla v|^2d\mu\big(g(0)\big)+B\int_Mv^2d\mu\big(g(0)\big).\nonumber\\
\label{44}
\end{eqnarray}

Recall that  $\lambda_0$ is the first eigenvalue of $F$ entropy as characterized in (\ref{Fentropy}), i.e.
\begin{equation}
\lambda_0=\inf_{\|v\|_2=1}\int_M(4|\nabla v|^2+Sv^2)d\mu\big(g(0)\big).
\end{equation}
This eigenvalue has been studied widely and is a very powerful tool for the understanding of Riemannian manifolds. One can find more details in \cite{LJ}.
If $\lambda_0>0$, by  Sobolev inequality (\ref{44}), we know that  the assumption of Sobolev inequality in Theorem \ref{th1} at initial time holds with $B_0=0$, i.e.
\begin{eqnarray}
\Big(\int_M v^{\frac{2n}{n-2}}d\mu\big(g_0\big)\Big)^{\frac{n-2}{n}}&\leq& \tilde{A}_0 \int_M\big(|\nabla v|^2+\frac{1}{4}Sv^2\big)d\mu\big(g_0\big)\label{45}
\end{eqnarray}
where $\tilde{A}_0$ depends only on initial metric $g_0$ and the first eigenvalue of $F$ entropy $\lambda_0$. Furthermore, the log-Sobolev inequality (\ref{S3}) in Lemma \ref{lm2} holds with $B_0=0$. Therefore from Theorem \ref{th1}, we have the following corollary:
\begin{corollary}\label{co1}
Let  $(g, \phi)(t)$ be a solution of the system (\ref{O1}). Assume $\lambda_0>0$.
 Then there exists a positive number $A_0$, depending only on initial metric $g_0$ and $\lambda_0$,  such that for all $v\in W^{1,2}(M,g(t))$, $t\in[0,T_0)$, it holds that
\begin{eqnarray}
\Big(\int_M v^{\frac{2n}{n-2}}d\mu\big(g(t)\big)\Big)^{\frac{n-2}{n}}&\leq& A_0 \int_M \big(|\nabla v|^2+\frac{1}{4}Sv^2\big)d\mu\big(g(t)\big).\label{46}
\end{eqnarray}
\end{corollary}

\section{The proof of Theorem \ref{th2}}
\setcounter{equation}{0}
In this section, we establish a certain Gaussian type upper bound for the fundamental solutions of the conjugate heat equation with proper lower bound of Ricci curvature. This Gaussian upper bound in Ricci flow was proved by Zhang in \cite{Z1} with the assumption of the lower bound of injectivity with the help of  Sobolev inequality by Heybey (c.f. in \cite{Heb}), but here by using the uniform Sobolev inequality in corollary 4.1, we derive the similar Gaussian upper bound  without any assumption of the lower bound of injectivity. To prove the Gaussian upper bound, we need the following interpolation theorem first. The proof is standard by using maximum principle (see also \cite{LY, MH, N2, LK}).
\begin{theorem}
Let $(g, \phi)(t)$ be a solution of (\ref{O1}) and $u(x,t)$ be a positive solution to heat equation
\begin{equation}
\triangle u-\partial_tu=0 \label{N1}
\end{equation}
for $t\in [0,T]$. Then it holds that
\begin{equation}
\frac{|\nabla u(x,t)|}{u(x,t)}\leq \sqrt{\frac{1}{t}}\sqrt{\ln \frac{A}{u(x,t)}}\label{PI1}
\end{equation}
for $(x,t)\in M\times[0,T]$. Here $A=\sup_{M\times[0,T]}u$. \\
Moreover, for any $\delta>0$, $x,y\in M$ and $0<t<T$, the following interpolation inequality holds
\begin{equation}
u(y,t)\leq A^{\frac{\delta}{1+\delta}}u^{\frac{1}{1+\delta}}(x,t)\exp(\frac{d^2(x,y,t)}{4t\delta}).\label{PI2}
\end{equation}\label{Thm5}
\end{theorem}
\begin{proof}
Using the heat equation (\ref{N1}), we have the following computations:
\begin{eqnarray}
(\triangle-\partial_t)(u\ln(\frac{A}{u}))&=&\triangle u \ln\frac{A}{u}+u\triangle(\ln\frac{A}{u})+2\nabla u\nabla \ln \frac{A}{u}-\partial_tu \ln\frac{A}{u}-u\partial_t(\ln\frac{A}{u})\nonumber\\
&=&\triangle u\ln \frac{A}{u}+u(-\frac{\triangle u}{u}+\frac{|\nabla u|^2}{u^2})-2\frac{|\nabla u|^2}{u}
-\triangle u\ln\frac{A}{u}+\partial_t u\nonumber\\
&=&-\frac{|\nabla u|^2}{u},\label{PI3}
\end{eqnarray}
\begin{eqnarray}
\triangle(\frac{|\nabla u|^2}{u})&=&\frac{\triangle |\nabla u|^2}{u}+\triangle(\frac{1}{u})|\nabla u|^2+2\nabla|\nabla u|^2\nabla(\frac{1}{u})\nonumber\\
&=&\frac{\triangle |\nabla u|^2}{u}+(\frac{2|\nabla u|^2}{u^3}-\frac{\triangle u}{u^2})|\nabla u|^2
-4\frac{u_iu_ju_{ij}}{u^2},\label{PI4}
\end{eqnarray}
and
\begin{eqnarray}
\partial_t(\frac{|\nabla u|^2}{u})=\frac{\partial_t|\nabla u|^2}{u}-\frac{|\nabla u|^2}{u^2}\partial_tu
=\frac{2\langle\nabla u,\nabla \triangle u\rangle+2S_{ij}u_iu_j}{u}-\frac{|\nabla u|^2}{u^2}.\label{PI5}
\end{eqnarray}
Combining (\ref{PI4}) and (\ref{PI5}) together, we derive
\begin{eqnarray}
(\triangle -\partial_t)(\frac{|\nabla u|^2}{u})&=&\frac{\triangle |\nabla u|^2}{u}+\frac{2|\nabla u|^4}{u^3}-4\frac{u_iu_ju_{ij}}{u^2}-\frac{2\langle\nabla u,\nabla \triangle u\rangle+2S_{ij}u_iu_j}{u}\nonumber\\
&=&\frac{2u_{ij}^2+4u_iu_j\phi_i\phi_j}{u}+\frac{2|\nabla u|^4}{u^3}-4\frac{u_iu_ju_{ij}}{u^2}\nonumber\\
&=&\frac{4}{u}|d\phi(\nabla u)|^2+\frac{2}{u}|u_{ij}-\frac{u_iu_j}{u}|^2.\label{PI6}
\end{eqnarray}
Combining (\ref{PI3}) and (\ref{PI6}), we have
\begin{eqnarray}
(\triangle -\partial_t)(\frac{t|\nabla u|^2}{u}-u\ln\frac{A}{u})&=&-\frac{|\nabla u|^2}{u}+\frac{4}{u}|d\phi(\nabla u)|^2
+\frac{2}{u}|u_{ij}-\frac{u_iu_j}{u}|^2+\frac{|\nabla u|^2}{u}\nonumber\\
&=&\frac{4}{u}|d\phi(\nabla u)|^2+\frac{2}{u}|u_{ij}-\frac{u_iu_j}{u}|^2\nonumber\\
&\geq&0.\label{PI7}
\end{eqnarray}
Notice that $A=\sup_{M\times[0,T]}u$, then at $t=0$ we have
\[\frac{t|\nabla u|^2}{u}-u\ln\frac{A}{u}=-u\ln\frac{A}{u}\leq0.\]
From (\ref{PI7}), the maximum principle implies that
$$
\frac{t|\nabla u|^2}{u}-u\ln\frac{A}{u}\leq0,
$$
which gives (\ref{PI1}).

Set
$
\ell(x,t)=\ln\frac{A}{u(x,t)},
$
then inequality (\ref{PI1}) yields
\[|\nabla \sqrt{\ell(x,t)}|=\frac{1}{2}|\frac{\nabla u}{u \sqrt{\ell}}|\leq\frac{1}{\sqrt{4t}}.\]
Thus, for any $x,y\in M$, we can integrate the above inequality along a minimizing geodesic joining $x$ and $y$ to get
$$
\sqrt{\ln\frac{A}{u(x,t)}}\leq\sqrt{\ln\frac{A}{u(y,t)}}+\frac{d(x,y,t)}{\sqrt{4t}},
$$
Thus, for any $\delta >0$ we have
\begin{eqnarray*}
\ln\frac{A}{u(x,t)}&\leq&\ln\frac{A}{u(y,t)}+\frac{d^2(x,y,t)}{4t}+\sqrt{\ln\frac{A}{u(y,t)}}\frac{d(x,y,t)}{\sqrt{t}}\nonumber\\
&\leq&\ln\frac{A}{u(y,t)}+\frac{d^2(x,y,t)}{4t}+\delta\ln\frac{A}{u(y,t)}+\frac{d^2(x,y,t)}{4t\delta},
\end{eqnarray*}
which immediately  gives (\ref{PI2}).
\end{proof}

Now we turn to  prove Theorem \ref{th2}. With the help of the uniform Sobolev inequality in Corollary \ref{co1} and the above interpolation theorem, we will begin with the traditional method of establishing a mean value inequality via Moser's iteration and a weighted estimate in the spirit of Davies in \cite{Da} to give the full proof of the Gaussian type upper bound (see also \cite{L,Z2}).  
\begin{proof}[Proof of Theorem \ref{th2}]
We divide the proof into two steps.\\
{\em Step 1.} By using Morse's iteration, we prove a mean value inequality for the positive solution $u$ of conjugate equation (\ref{O2}).\\
For any arbitrary constant $p\geq1$, it follows trivially that
\begin{equation}
\Delta u^p-pSu^p+\partial_t u^p\geq0.\label{BRF2}
\end{equation}
Define the region
\[Q_{\sigma r}:=\left\{(y,s)\mid y\in M, t\leq s\leq t+(\sigma r)^2, d(x,y,s)\leq \sigma r\right\},\]
with $r>0$, $1<\sigma \leq 2$.

Let $\varphi(\rho): [0,+\infty)\rightarrow [0,1]$ be a smooth function satisfying: $|\varphi'|\leq\frac{2}{(\sigma -1)r}$, $\varphi'\leq0$, $\varphi\geq0$, $\varphi(\rho)=1$ when $0\leq\rho\leq r$, and $\varphi(\rho)=0$ when $\rho\geq \sigma r$. Let $\eta(s): [0,+\infty)\rightarrow [0,1]$ be a smooth function satisfying: $|\eta'|\leq\frac{2}{(\sigma -1)^2r^2}$, $\eta'\leq0$, $\eta\geq0$, $\eta(s)=1$ when $ s\leq t+r^2$, and $\eta(s)=0$ when $t+(\sigma r)^2\leq s\leq T$.
Define a cut-off function $\psi(y,s)$ by
\[\psi(y,s)=\varphi(d(y,x,s))\eta(s).\]
Writing $\omega=u^p$, multiplying $\omega \psi^2$ to (\ref{BRF2}) for $p\geq1$ and integrating by parts yields
\begin{eqnarray}
\int_{Q_{\sigma r}}\nabla(\omega \psi^2)\nabla\omega dg(y,s)ds+p\int_{Q_{\sigma r}}S\omega^2 \psi^2dg(y,s)ds
\leq\int_{Q_{\sigma r}} (\partial_s\omega) \omega\psi^2dg(y,s)ds. \nonumber\\
\label{BRF3}
\end{eqnarray}
Integrating by parts, the right hand side of (\ref{BRF3}) yields
\begin{eqnarray*}
\int_{Q_{\sigma r}}(\partial_s\omega)\omega\psi^2dg(y,s)ds&=&-\int_{Q_{\sigma r}}\omega^2\psi\partial_s\psi dg(y,s)ds +\frac{1}{2}\int_{Q_{\sigma r}} (\psi\omega)^2Sdg(y,s)ds
\nonumber\\&&-\frac{1}{2}\int_{B_{\sigma r}(t)}(\psi\omega)^2dg(y,t).
\end{eqnarray*}
By the non-negativity of $Sy$ and using the identity (c.f. \cite{CLN, BL})
\begin{equation*}
\partial_s d(x,y,s)=-\int_0^{d(x,y,s)} Sy(\gamma'(\tau),\gamma'(\tau))d\tau\leq0,
\end{equation*}
we have
\[\partial_s\psi=\eta(s)\varphi'(d(y,x,s))\partial_sd(x,y,s)+\varphi(d(y,x,s))\eta'(s)\geq\varphi(d(y,x,s))\eta'(s).\]
Hence
\begin{eqnarray}
\int_{Q_{\sigma r}}(\partial_s\omega)\omega\psi^2dg(y,s)ds&\leq&\int_{Q_{\sigma r}}\omega^2\psi\varphi(d(y,x,s))|\eta'(s)| dg(y,s)ds\nonumber\\
&&+\frac{1}{2}\int_{Q_{\sigma r}} (\psi\omega)^2Sdg(y,s)ds-\frac{1}{2}\int_{B_{\sigma r}(t)}(\psi\omega)^2dg(y,t). \label{BRF6}
\end{eqnarray}
By direct calculation, we get
\begin{eqnarray}
\int_{Q_{\sigma r}}\nabla(\omega \psi^2)\nabla\omega dg(y,s)ds=\int_{Q_{\sigma r}}|\nabla(\omega \psi)|^2dg(y,s)ds
-\int_{Q_{\sigma r}} |\nabla\psi|^2\omega^2dg(y,s)ds. \nonumber\\\label{BRF4}
\end{eqnarray}
From (\ref{BRF3}), (\ref{BRF6}), (\ref{BRF4}), and  the assumptions of $p\geq1$ and $S\geq0$, we deduce
\begin{eqnarray}
&&\int_{Q_{\sigma r}}|\nabla(\omega\psi)|^2dg(y,s)ds+\frac{1}{2}\int_{Q_{\sigma r}}S(\omega\psi)^2+\frac{1}{2}\int_{B_{\sigma r}(t)}(\psi\omega)^2dg(y,t)\nonumber\\
&&\leq\int_{Q_{\sigma r}}\omega^2\psi\varphi(d(y,x,s))|\eta'(s)| dg(y,s)ds+\int_{Q_{\sigma r}}|\nabla \psi|^2\omega^2dg(y,s)ds\nonumber\\
&&\leq\frac{c}{(\sigma-1)^2r^2}\int_{Q_{\sigma r}}\omega^2dg(y,t).\label{BRF7}
\end{eqnarray}
Using H\"{o}lder inequality, we get
\begin{equation}
\int(\psi\omega)^{2(1+\frac{2}{n})}dg\leq(\int(\psi\omega)^{\frac{2n}{n-2}}dg)^{\frac{n-2}{n}}(\int(\psi\omega)^2dg)^{\frac{2}{n}}.
 \label{BRF8}
\end{equation}
From Corollary \ref{co1}, we know that for any $t\in(0,T)$, the following Sobolev imbedding inequality holds:
\begin{equation}
\left(\int(\psi\omega)^{\frac{2n}{n-2}}dg(s)\right)^{\frac{n-2}{n}}\leq A_0\int\left(|\nabla(\psi\omega)|^2+S(\psi\omega)^2\right)dg(s),
\label{BRF9}\end{equation}
where $A_0$ depends  only on dimension $n$, $\lambda_0$ and initial metric $g_0$.\\
Substituting (\ref{BRF9}) to (\ref{BRF8}), we get
\begin{equation*}
\int_{B_{\sigma r}(s)}(\psi\omega)^{2(1+\frac{2}{n})}dg(s)\leq A_0\int_{B_{\sigma r}(s)}\left(|\nabla(\psi\omega)|^2+S(\psi\omega)^2\right)dg(\int_{B_{\sigma r}(s)}
(\psi\omega)^2dg)^{\frac{2}{n}}
\end{equation*}
Setting $\theta=1+\frac{2}{n}$, integrating the above inequality with respect to $s$ on $[t,t+(\sigma r)^2]$ and using (\ref{BRF7}), we reach
\begin{equation*}
\int_{Q_{\sigma r}}(\psi\omega)^{2\theta}dg(y,s)ds\leq A_0\big(\frac{1}{(\sigma-1)^2r^2}\int_{Q_{\sigma r}(x,t)}\omega^2dg(y,s)ds\big)^\theta.
\end{equation*}
By the definition of $\psi$, we conclude 
\begin{equation}
\int_{Q_r}\omega^{2\theta}dg(y,s)ds\leq A_0\big(\frac{1}{(\sigma-1)^2r^2}\int_{Q_{\sigma r}(x,t)}\omega^2dg(y,s)ds\big)^\theta.
\label{BRF11}
\end{equation}
Now we choose the sequences of $\sigma_i$ and $p_i$ as
\[\sigma_0=2, \sigma_i=2-\sum_{j=1}^i2^{-j}, p_i=\theta^i,\]
then inequality (\ref{BRF11}) gives that
$$
\|u^2\|_{L^{\theta^{i+1}}(\sigma_{i+1}r)}\leq A_0^{(\frac{1}{\theta})^{i+1}}(\frac{\sigma^2_{i+1}}{(\sigma_i-\sigma_{i+1})^2r^2})^{(\frac{1}{\theta})
^i}\|u^2\|_{L^{\theta^{i}}(\sigma_{i}r)},
$$
which gives a $L^2$ mean value inequality
\begin{equation}
\sup_{Q_{r/2}(x,t)}u^2\leq\frac{c}{r^{2+n}}\int_{Q_r(x,t)}u^2dg(y,s)ds.\label{BRF12}
\end{equation}
Here constant $c$ depends on dimension $n$, $\lambda_0$  and initial metric $g(0)$.
From here, by a generic trick of Li and Schoen ( c.f. in \cite{L, LS}), we arrive at a $L^1$ mean value inequality: for $r>0$,
\begin{equation}
\sup_{Q_{r/2}(x,t)}u\leq\frac{c}{r^{2+n}}\int_{Q_r(x,t)}udg(y,s)ds.\label{BRF13}
\end{equation}

For  $y\in M$ and $s>t$,  applying (\ref{BRF13}) on $u=G(\cdot,\cdot : y,T)$ with $r=\sqrt{\frac{T-t}{2}}$ and using the fact 
$\int_Mu(z,\tau)dg(z,\tau)d\tau=1$,
we conclude
\begin{equation}
G(x,t;y,T)\leq\frac{c}{(T-t)^\frac{n}{2}}. \label{BRF14}
\end{equation}

{\em  Step 2.} Using methods of the exponential weight due to Davies in \cite{Da} to prove the full bound with the exponential term.\\
\text{\quad}It is clear that we only have to deal with the case $d(x_0,y_0,T)\geq 2\sqrt{T-t}$. Otherwise, by (\ref{BRF14}), the Gaussian type upper bound (\ref{BRF1}) holds obviously. Pick a point $x_0\in M$, a number $\lambda<0$ which is determined later and a function $f\in L^2(M,g(T))$. Consider the functions $u(x,t)$ and $F(x,t)$ defined by
\begin{eqnarray*}
u(x,t)&=:&\int_M G(x,t;y,T)e^{-\lambda d(y,x_0,T)}f(y)dg(y,T)\\
F(x,t)&=:&e^{\lambda d(x,x_0,t)}u(x,t).
\end{eqnarray*}
It is clear that $u$ is a solution of (\ref{O2})
with initial date
\[u(x,T)=e^{-\lambda d(x,x_0,T)}f(x).\]
Direct calculation shows
\begin{eqnarray*}
\partial_t\int F^2(x,t)dg(x,t)&=&\partial_t\int e^{2\lambda d(x.x_0,t)}u^2(x,t)dg(x,t)\\
&=&2\lambda\int e^{2\lambda d(x.x_0,t)}\partial_t d(x,x_0,t)u^2(x,t)dg(x,t)\\
&&-\int e^{2\lambda d(x.x_0,t)}u^2(x,t)S(x,t)dg(x,t)\\
&&-2\int e^{2\lambda d(x.x_0,t)}u(x,t)[\triangle u-Su]dg(x,t)\\
&\geq&-2\int e^{2\lambda d(x.x_0,t)}u(x,t)\triangle udg(x,t),
\end{eqnarray*}
where the last inequality has used the assumption that $Sy\geq0$ , $\lambda<0$, and $\partial_td(x,x_0,t)\leq0$.\\
Using integration by parts, we turn this inequality into
\begin{eqnarray*}
\partial_t\int F^2(x,t)dg(x,t)&\geq&4\lambda\int e^{2\lambda d(x.x_0,t)}u(x,t)\nabla u\nabla d(x,x_0,t)dg(x,t)\\
&&+2\int e^{2\lambda d(x.x_0,t)}|\nabla u|^2dg(x,t).
\end{eqnarray*}
Direct calculation gives
\begin{eqnarray*}
\int |\nabla F(x,t)|^2dg(x,t)&=&\int |\nabla(u(x,t)e^{\lambda d(x.x_0,t)})|^2dg(x,t)\\
&=&\int e^{2\lambda d(x.x_0,t)}|\nabla u|^2dg(x,t)\\
&&+2\lambda\int e^{2\lambda d(x.x_0,t)}u(x,t)\nabla u\nabla d(x,x_0,t)dg(x,t)\\
&&+\lambda^2\int e^{2\lambda d(x.x_0,t)}|\nabla d|^2u^2dg(x,t).
\end{eqnarray*}
Combining the last two expressions, we deduce
\begin{eqnarray*}
\partial_t\int F^2(x,t)dg(x,t)\geq2\int |\nabla F(x,t)|^2dg(x,t)-2\lambda^2\int e^{2\lambda d(x.x_0,t)}u^2dg(x,t),
\end{eqnarray*}
which shows
\[\partial_t\int F^2(x,t)dg(x,t)\geq -2\lambda^2\int F^2(x,t)dg(x,t).\]
Integrating on $[t,T]$, we arrive at the following $L^2$ estimate
\begin{eqnarray}
\int F^2(x,t)dg(x,t)\leq e^{2\lambda^2(T-t)}\int F^2(x,T)dg(x,T)
=e^{2\lambda^2(T-t)}\int f^2(x)dg(x,T).\nonumber\\
\label{BRF21}
\end{eqnarray}
Therefore, by the mean value inequality (\ref{BRF12}) with $r=\sqrt{\frac{T-t}{2}}$, the following holds
\begin{eqnarray*}
u^2(x,t)&\leq&\frac{c}{(T-t)^{1+\frac{n}{2}}}\int^{\frac{T+t}{2}}_{t}\int_{B(x,\sqrt{\frac{T-t}{2}},\tau)}u^2(z,\tau)dg(z,\tau)d\tau\\
&\leq&\frac{c}{(T-t)^{1+\frac{n}{2}}}\int^{\frac{T+t}{2}}_t\int_{B(x,\sqrt{\frac{T-t}{2}},\tau)}e^{-2\lambda d(z,x_0,\tau)}F^2(z,\tau)dg(z,\tau)d\tau.
\end{eqnarray*}
In particular, this holds at $x=x_0$. Therefore, by the assumption that $\lambda <0$, we get
\[u^2(x_0,t)\leq\frac{ce^{-2\lambda\sqrt{\frac{T-t}{2}}}}{(T-t)^{1+\frac{n}{2}}}\int^{\frac{T+t}{2}}_{t}\int_{B(x_0,\sqrt{\frac{T-t}{2}},\tau)}F^2(z,\tau)dg(z,\tau)d\tau.\]
From (\ref{BRF21}), it follows that
\[u^2(x_0,t)\leq\frac{ce^{2\lambda^2(T-t)-2\lambda\sqrt{\frac{T-t}{2}}}}{(T-t)^{\frac{n}{2}}}\int f^2(y)dg(y,T).\]
i.e.
\begin{eqnarray}
\left(\int G(x_0,t;z,T)e^{-\lambda d(z,x_0,T)}f(z)dg(z,T)\right)^2
\leq\frac{ce^{2\lambda^2(T-t)-2\lambda\sqrt{\frac{T-t}{2}}}}{(T-t)^{\frac{n}{2}}}\int f^2(y)dg(y,T).
\nonumber\\\label{BRF22}\end{eqnarray}
Now we fix $y_0$ such that $d(y_0,x_0,T)^2\geq 4(T-t)$. Then it is clear that, by $\lambda<0$ and the triangle inequality,
\[-\lambda d(z,x_0,T)\geq -\frac{\lambda}{2}d(x_0,y_0,T)\]
provided by  $d(z,y_0,T)\leq \sqrt{T-t}$. Then (\ref{BRF22}) implies
\begin{eqnarray*}
(\int_{B(y_0,\sqrt{T-t},T)} G(x_0,t;z,T)f(z)dg(T))^2
\leq\frac{ce^{\lambda d(x_0,y_0,T)+2\lambda^2(T-t)-2\lambda\sqrt{\frac{T-t}{2}}}}{(T-t)^{\frac{n}{2}}}\int f^2(y)dg(T).
\end{eqnarray*}
From the Cauchy-Schwarz inequality, it follows trivially that 
\begin{eqnarray*}
2\lambda^2(T-t)-2\lambda\sqrt{\frac{T-t}{2}}\leq3\lambda^2(T-t)+\frac{1}{2}.
\end{eqnarray*}
If we choose
$\lambda=-\frac{d(x_0,y_0,T)}{b(T-t)}$,
then we have
\begin{eqnarray*}
\big(\int_{B(y_0,\sqrt{T-t},T)} G(x_0,t;z,T)f(z)dg(z,T)\big)^2
\leq\frac{ce^{-\frac{c_1d^2(x_0,y_0,T)}{T-t}}}{(T-t)^{\frac{n}{2}}}\int f^2(y)dg(y,T)
\end{eqnarray*}
with $b>0$ sufficiently large, and $c_1$ is an absolute constant.
Then by arbitrariness of $f$, we derive 
\[\int_{B(y_0,\sqrt{T-t},T)} G^2(x_0,t;z,T)dg(z,T)\leq\frac{ce^{-c_1\frac{d^2(x_0,y_0,T)}{T-t}}}{(T-t)^{\frac{n}{2}}}.\]
Hence, there exists $z_0\in B(y_0,\sqrt{T-t},T)$ such that
\begin{equation}
G^2(x_0,t;z_0,T)\leq\frac{ce^{-c_1\frac{d^2(x_0,y_0,T)}{T-t}}}{(T-t)^{\frac{n}{2}}|B(y_0,\sqrt{T-t},T)|_T}.
\label{BRF24}
\end{equation}
Finally, let us remind that in \cite{CM} the adjoint property of the $G(x_0,t: \cdot,\cdot)$ is obtained, thus we have
\[\triangle_zG(x,t;z,\tau)-\partial_\tau G(x,t;z,\tau)=0\]
along the extend  Ricci flow (\ref{O1}).
Choosing $\delta=1$, it then follows from Theorem \ref{Thm5} that
\begin{eqnarray}
G(x_0,t;y_0,T)\leq \sqrt{G(x_0,t;z_0,T)}\sqrt{A}e^{\frac{d^2(y_0,z_0,T)}{4(T-t)}}
\leq e^{1/4}\sqrt{G(x_0,t;z_0,T)}\sqrt{A},\label{BRF25}
\end{eqnarray}
where $A=\sup_{M\times[\frac{t+T}{2},T]}G(x_0,t;\cdot,\cdot)$ and we used $z_0\in B(y_0,\sqrt{T-t},T)$. 
Since (\ref{BRF14}) implies $A\leq\frac{c}{(T-t)^{\frac{n}{2}}}$,
then (\ref{BRF24}) and (\ref{BRF25}) immediately yields
\[G(x_0,t;y_0,T)^2\leq\frac{c}{(T-t)^{\frac{n}{2}}}\frac{1}{(T-t)^{\frac{n}{4}}
\sqrt{|B(y_0,\sqrt{T-t},T)|_T}}e^{-\frac{c_1d^2(x_0,y_0,T)}{T-t}}.\]
Therefore, by Cauchy-Schwarz inequality, we get
\begin{eqnarray*}
G(x_0,t;y_0,T)&\leq& c(\frac{1}{(T-t)^{\frac{n}{2}}}+\frac{1}{|B(y_0,\sqrt{T-t},T)|_{T}})e^{-\frac{c_1d^2(x_0,y_0,T)}{T-t}}\\
&\leq&\frac{c}{|B(y_0,\sqrt{T-t},T)|_{T}}e^{-\frac{c_1d^2(x_0,y_0,T)}{T-t}}.
\end{eqnarray*}
Here $c$ depends on dimension $n$, $\lambda_0$  and initial metric $g(0)$ and $c_1$ depends only on dimension $n$.
In last inequality, we used volume comparison theorem with the non-negative Ricci curvature.
Since $x_0$ and $y_0$ are arbitrary, we complete the proof.
\end{proof}


\begin{thebibliography}{50}
\bibitem{CM}{\sc M. Christine}, \textit{The fundamental solution on manifolds with time-dependent metrics}, J. Geom. Anal. \textbf{12} (2002), no. 3, 425--436.
    
\bibitem{CLN}{\sc B. Chow, P. Lu and L. Ni}, \textit{Hamilton¡¯s Ricci flow}, Graduate Studies in Mathematics, vol. 77, American
              Mathematical Society, Providence, RI, 2006.

\bibitem{Da}{\sc E. B. Davies}, \textit{Heat kernel and spectral theory}, Cambridge University, Press, 1989.

\bibitem{Ham1}{\sc R. S. Hamilton}, \textit{Three-mainfolds with positive Ricci curvature}, J. Differ. Geom. \textbf{17} (1982),  255--306.

\bibitem{Ham2}{\sc R. S. Hamilton}, \textit{Four manifolds with positive Ricci curvature}, J. Differ. Geom. \textbf{24} (1986), 153--179.

\bibitem{Ham3}{\sc R. S. Hamilton}, \textit{The Harnack estimate for the Ricci flow}, J. Differ. Geom. \textbf{37} (1993), 225--243.

\bibitem{Heb}{\sc E. Hebey}, \textit{Optimal Sobolev inequalities on complete Riemannian manifolds with Ricci curvature bounded below and positive injectivity radius}, Amer. J. Math. \textbf{118} (1996),no. 2, 291--300.

\bibitem{BL}{\sc B. List}, \textit{Evolution of an extended Ricci flow system}, Comm. Anal. Geom. \textbf{16} (2008), no. 5,
             1007--1048.
\bibitem{LJ}{\sc J. F. Li}, \textit{Eigenvalues and energy functionals with monotonicity formulae under Ricci flow}, Math. Ann.  \textbf{338} (2007), no. 4, 927--946.

\bibitem{KZ}{\sc S. L. Kuang and Q. S. Zhang}, \textit{A gradient estimate for all positive solutions of the conjugate heat equation under Ricci flow}, J. Funct. Anal. \textbf{255} (2008), no. 4, 1008--1023.

\bibitem{Sal}{\sc S. C. Laurent}, \textit{Aspects of Sobolev-type inequalities}, London Mathematical Society Lecture Note Series, 289. Cambridge University Press, Cambridge, 2002.

\bibitem{L} {\sc P. Li}, \textit{ Geometric analysis.} Cambridge Studies in Advanced Mathematics, 134. Cambridge University Press, Cambridge, 2012. x+406 pp.

\bibitem{LY}{\sc P. Li and S. T. Yau}, \textit{On the parabolic kernel of the Schr\"{o}dinger operator}, Acta Math. \textbf{156} (1986),
             no. 3¨C4, 153-¨C201.
\bibitem{LS}{\sc P. Li and R. Schoen}, \textit{$L^p$ and mean value properties of subharmonic functions on Riemannian manifolds}, Acta Math., \textbf{153} (1984), no.3-4, 279--301.

\bibitem{MH}{\sc M. Bailesteanu and H. Tran}, \textit{Heat kernel estimates under the Ricci-harmonic map flow}, arXiv:math.DG/ 1310.1619

\bibitem{MR}{\sc R. M\"{u}ller}, \textit{Ricci flow coupled with harmonic map flow}, Ann. Sci. Ecole Norm. S. \textbf{45} (2012),no. 1, 101--142.

\bibitem{N1}{\sc L. Ni}, \textit{The entropy formula for linear heat equation}, J. Geom. Anal. \textbf{14} (2004),
             87--100.
\bibitem{N2}{\sc L. Ni}, \textit{A note on Perelman's Li-Yau-Hamilton inequality}, Comm. Anal. Geom. \textbf{14} (2006),
             883--905.

\bibitem{P1}{\sc G. Perelman}, \textit{The entropy formula for the Ricci flow and its geometric applications},
              arXiv:math.DG/ 0211159.



\bibitem{LK}{A. Q. Zhu}, \textit{Differential Harnack inequalities for the backward heat equation with potential
          under the harmonic-Ricci flow}, J. Math. Anal. Appl. \textbf{406} (2013),no. 2, 502--510.

\bibitem{Z1}{\sc Q. S. Zhang}, \textit{Some gradient estimates for the heat equation on domain for an equation by
    Perelman}, IMRN, \textbf{2006} (2006), article id: 92314, 1¨C-39.

\bibitem{Z2}{\sc Q. S. Zhang}, \textit{Sobolev inequality under Ricci flow}, IMRN, \textbf{2007} (2007),article id:
            rnm056, 1--17

\bibitem{Z3}{\sc Q. S. Zhang}, \textit{Sobolev inequalities, heat kernels under Ricci flow and Poincar\'{e} conjecture},CRC Press,
                   Boca Raton, FL, 2010.

\end{thebibliography}
\end{document}